\documentclass[12pt,a4paper]{amsart}
\usepackage {xypic, verbatim, amscd, amsmath, amssymb,graphicx, colordvi, color }
\usepackage[all,2cell]{xy}
\usepackage{epsf}
\usepackage{latexsym}

\newcommand{\frg} {\mathfrak{g}}

\newcommand{\til} {\widetilde}

\newcommand{\vect} {\vec}
\newcommand{\la} {\langle}
\newcommand{\ra} {\rangle}

\newcommand {\geg} {\mathfrak{g}}
\newcommand {\heh} {\mathfrak{h}}

\newcommand {\wgeg} {\widehat{\mathfrak{g}}}

\newcommand {\ZZ} {\mathbb{Z}}

\newcommand {\CC} {\mathbb{C}}

\newtheorem {prop}{Proposition}

\newtheorem {thm}{Theorem}
\newtheorem {Lemma}{Lemma }
\newtheorem {definition}{Definition}

\newtheorem {Rem}{Remark}
\newtheorem {cor}{Corollary}
\newcommand {\Res} {\mathop{\mathrm{Res}}\displaylimits}

\begin{document}
\title[Diagram automorphisms]{Diagram automorphisms and Rank-Level duality }
\author{Swarnava Mukhopadhyay}
\thanks{The author was partially supported by NSF grant  DMS-0901249.}

\address{Department of Mathematics\\ UNC-Chapel Hill\\ CB \#3250, Phillips Hall
\\ Chapel Hill, NC 27599}
\email{swarnava@live.unc.edu}
\begin{abstract}
We study the effect of diagram automorphisms on rank-level duality. We use it to prove new symplectic rank-level dualities on genus zero smooth curves with marked points and chosen coordinates. We also show that rank-level dualities for the pair $\mathfrak{sl}(r), \mathfrak{sl}(s)$ in genus $0$ arising from representation theory can also be obtained from the parabolic strange duality proved by R. Oudompheng.
\end{abstract}
\maketitle
\section{Introduction}

Let $\frg$ be a finite dimensional simple Lie algebra, $\ell$ a non-negative integer called the level and $\vec{\lambda}=(\lambda_1,\dots,\lambda_n)$ a $n$-tuple of dominant weights of $\geg$ of level $\ell$. Consider $n$ distinct points $\vec{z}$ on $\mathbb{P}^1$ with coordinates $\xi_1,\dots, \xi_n$ and let $\mathfrak{X}$ denote the corresponding data. One can associate a finite dimensional complex vector space $\mathcal{V}^{\dagger}_{\vec{\lambda}}(\mathfrak{X},\geg,\ell)$ to this data. These spaces are known as conformal blocks. The dual space of conformal blocks is denoted by $\mathcal{V}_{\vec{\lambda}}(\mathfrak{X},\geg,\ell)$. More generally one can define conformal block on curves of arbitrary genus. We refer the reader to ~\cite{TUY} for more details. Rank-level duality is a duality between conformal blocks associated to two different Lie algebras.

A diagram automorphism of a Dynkin diagram is a permutation of its nodes which leaves the diagram invariant. For every digram automorphism, we can construct a finite order automorphism of the Lie algebra. These automorphisms are known as outer automorphisms. In the following, we restrict ourselves to affine Kac-Moody Lie algebras $\wgeg$ (see Section 2) and to those diagram automorphisms which corresponds to the center $Z(G)$ of the simply connected group $G$ associated to a finite dimensional Lie algebra $\geg$.  

In ~\cite{FS}, the conformal blocks $\mathcal{V}^{\dagger}_{\vec{\lambda}}(\mathfrak{X},\geg,\ell)$ and $\mathcal{V}^{\dagger}_{\vec{\omega}\vec{\lambda}}(\mathfrak{X},\geg,\ell)$ are identified via an isomorphism which is flat with respect to the KZ connection, where $\vec{\omega}=(\omega_1, \cdots, \omega_n)\in Z(G)^n $ such that $ \prod_{s=1}^n\omega_s= \text{id}$, $\vect{\omega}\vect{\lambda}=(\omega_1^* \lambda_1, \dots, \omega_n^* \lambda_n)$ and $\omega^*$ is a permutation of level $\ell$ weights $P_{\ell}(\geg)$ associated to a diagram automorphism $\omega$. The main purpose of this article is to understand single and multi-shift automorphisms as ``conjugations" and study the functoriality of the isomorphism in ~\cite{FS} under embeddings of Lie algebras. We refer the reader to Section \ref{conjugation} for more details. 

We briefly recall the construction of the above isomorphism in ~\cite{FS}. Let $P^{\vee}$ and $Q^{\vee}$ denote the coweight and the coroot lattice of $\geg$ respectively and $\Gamma_n^{\geg}=\{(\mu_1,\dots,\mu_n) | \mu_i \in P^{\vee} \text{and} \sum_{i=1}^n\mu_i=0\}$. For $\vec{z}\in \mathbb{P}^1$, a positive integer $s$ and $\vec{\mu} \in \Gamma_n^{\geg}$ a multi-shift automorphism $\sigma_{\vec{\mu},s}(\vec{z})$ of $\wgeg$ is constructed in ~\cite{FS}. Multi-shift automorphisms are generalization of single-shift automorphisms. We refer the reader to Section \ref{conjugation} for more details. 

Multi-shift automorphisms add up to give an automorphism $\sigma_{\vec{\mu}}(\vec{z})$ of the Lie algebra $\wgeg_n=\oplus_{i=1}^n\geg\otimes \mathbb{C}((\xi_i))\oplus \mathbb{C}.c$. To implement the action of $\sigma_{\vec{\mu}}(\vec{z})$ on tensor product of highest weight modules, a map from $\Theta_{\vec{\mu}}(\vec{z}): \otimes_{i=1}^n\mathcal{H}_{\lambda_i} \rightarrow \otimes_{i=1}^n\mathcal{H}_{\omega_i^*\lambda_i}$ is given, where $\mathcal{H}_{\lambda}$ is an integrable irreducible highest weight $\wgeg$-module of highest weight $\lambda$. Since the automorphism $\sigma_{\vec{\mu}}(\vec{z})$ preserves the current algebra $\geg\otimes \Gamma(\mathbb{P}^1- \vec{z},\mathcal{O}_{\mathbb{P}^1})$, the isomorphism of conformal blocks follows by taking coinvariants. Moreover the map $\Theta_{\vec{\mu}}(\vec{z})$ is chosen such that the induced map between the conformal blocks is flat with respect to the KZ connection. We now describe the main results of this paper. 

Let $G$ be a simply connected simple Lie group with Lie algebra $\geg$. We assume that $G$ is classical. Let $\mu \in P^{\vee}$, and consider $\tau_{\mu}= \exp(\ln{\xi}.\mu)$. It is well defined upon a choice of a branch of the complex logarithm but conjugation by $\tau_{\mu}$ is independent of the branch of the chosen logarithm. 

\begin{thm}\label{conjugation}
The map $x\rightarrow \tau_{\mu}x\tau_{\mu}^{-1}$ defines an automorphism $\varphi_{\mu}$ of the loop algebra $\frg\otimes \CC((\xi))$. Further more the automorphism $\varphi_{\mu}$ coincide with the single-shift automorphism $\sigma_{\mu}$ restricted to $\frg\otimes\CC((\xi))$.
\end{thm}

Since the extension $\wgeg$ of $\frg\otimes \CC((\xi))$ is an universal central extension, an immediate corollary of Theorem \ref{conjugation} is the following. We refer the reader to Section \ref{conjugation} for more details:
\begin{cor}
The automorphism $\sigma_{\mu}$ is the unique extension of $\varphi_{\mu}$ to $\wgeg$.
\end{cor}
We consider a map $\phi : G_1 \times G_2 \rightarrow G,$ where $G_1, G_2$ and $G$ are simple, simply connected complex Lie groups with Lie algebras $\geg_1$, $\geg_2$ and $\geg$ respectively. Let us also denote the map of the corresponding Lie algebras by $\phi$. We extend $\phi$ to a map $\widehat{\phi}:\wgeg_1\oplus \wgeg_2 \rightarrow \wgeg$.  We prove the following theorem:

\begin{thm}\label{extension1} 
Let $\heh_1$, $\heh_2$ and $\heh$ be Cartan subalgebras of $\geg_1$, $\geg_2$ and $\geg$ such that $\phi(\heh_1\oplus \heh_2) \subset \heh$.
Consider $\vect{\mu}=(\mu_1, \cdots, \mu_n)\in(P_1^{\vee})^{n}$ such that $\vect{\tilde{\mu}}=(\phi(\mu_1), \cdots, \phi(\mu_n))\in (P^{\vee})^n$, where $P_1^{\vee}$ and $P^{\vee}$ denote the coweight lattice of $(\geg_1, \heh_1)$ and $(\geg,\heh)$ respectively. Then the following diagram commutes:

$$\xymatrix{
\wgeg_1\oplus \wgeg_2 \ar[r]^{\widehat{\phi}} \ar[d]^{ \sigma_{\vect{\mu},s}\oplus \text{id}} & \wgeg\ar[d]^{\sigma_{\vect{\tilde{{\mu}}},s}}\\
\wgeg_1 \oplus \wgeg_2  \ar[r]^{\widehat{\phi}} & \wgeg.
}
$$
\end{thm}
Let $\ell=(\ell_1,\ell_2)$ be the Dynkin multi-index (see Section \ref{conjugation}) of the embedding $\phi$. Let $\vec{\lambda}=(\lambda_1,\dots, \lambda_n)$ and $\vec{\gamma}=(\gamma_1, \dots, \gamma_n)$ be $n$-tuples of weights of $\geg_1$ and $\geg_2$ of level $\ell_1$ and $\ell_2$. Let $\vec{\Lambda}=(\Lambda_1, \dots, \Lambda_n)$ be an $n$-tuple of level $1$ weights of $\geg$ such that for $1\leq i \leq n$, the module $\mathcal{H}_{\lambda_i}\otimes \mathcal{H}_{\gamma_i}$ appears in the branching of $\mathcal{H}_{\Lambda_i}$ with multiplicity one. We get a map of conformal blocks. This map is known as a rank-level duality map. We refer the reader to ~\cite{M} for more details. 
$$\Psi : \mathcal{V}_{\vec{\lambda}}(\mathfrak{X}, \geg_1,\ell_1)\otimes \mathcal{V}_{\vec{\gamma}}(\mathfrak{X},\geg_2,\ell_2)\rightarrow \mathcal{V}_{\vec{\Lambda}}(\mathfrak{X},\geg,1).$$ 

\begin{Rem}
The rank-level duality map in ~\cite{M} is defined only for conformal embeddings but the same definition will work for arbitrary embeddings of Lie algebras.
\end{Rem}
Let $\vect{\Omega}=(\Omega_1,\dots, \Omega_n)$, where for each $1\leq i \leq n$, $\Omega_i$ is the image of the diagram automorphisms $\omega_i$ under the embedding $\phi$. Combining Theorem \ref{extension1} with the isomorphism ( Proposition \ref{fs}) in ~\cite{FS}, we have the following important corollary:
\begin{cor}\label{ext}
 The following are equivalent:
\begin{enumerate} 
\item The map $\mathcal{V}_{\vec{\lambda}}(\mathfrak{X}, \geg_1,\ell_1)\otimes \mathcal{V}_{\vec{\gamma}}(\mathfrak{X},\geg_2,\ell_2)\rightarrow \mathcal{V}_{\vec{\Lambda}}(\mathfrak{X},\geg,1)$ is nondegenerate.  
\item The map $\mathcal{V}_{\vec{\lambda}}(\mathfrak{X}, \geg_1,\ell_1)\otimes \mathcal{V}_{\vec{\omega}\vec{\gamma}}(\mathfrak{X},\geg_2,\ell_2)\rightarrow \mathcal{V}_{\vec{\Omega}\vec{\Lambda}}(\mathfrak{X},\geg,1)$ is nondegenerate. 
\end{enumerate}
\end{cor}

We now restrict ourselves to two specific embeddings. First we consider the map of Lie algebras induced by the tensor product of vector spaces $\phi  : \mathfrak{sl}(r)\oplus \mathfrak{sl}(s) \rightarrow \mathfrak{sl}(rs)$. Let $\vect{\lambda} =(\lambda_1, \cdots, \lambda_n)$ be a $n$-tuple of dominant weights of $\mathfrak{sl}(r)$ of level $s$ and $\vect{\gamma}=(\gamma_1,\cdots, \gamma_n)$ be a $n$-tuple of dominant weights of $\mathfrak{sl}(s)$ of level $r$. Assume that $\mathcal{H}_{\lambda_i} \otimes \mathcal{H}_{\gamma_i}$ appears in the decomposition of $\mathcal{H}_{\Lambda_{k_i}}$ where $\Lambda_{k_i}$ is a level $1$ dominant weight of $\mathfrak{sl}(rs)$. Using the branching rules described in ~\cite{ABI}, we get a rank-level duality map
$$\Psi :\mathcal{V}_{\vect{\lambda}}(\mathfrak{X}, \mathfrak{sl}(r),s)\otimes \mathcal{V}_{\vect{\gamma}}(\mathfrak{X}, \mathfrak{sl}(s),r) \rightarrow \mathcal{V}_{\vect{\Lambda}}(\mathfrak{X}, \mathfrak{sl}(rs),1),$$ where $\vec{\Lambda}=(\Lambda_{k_1}, \dots, \Lambda_{k_n})$. When $\dim_{\CC}\mathcal{V}_{\vect{\Lambda}}(\mathfrak{X}, \mathfrak{sl}(rs),1)=1$, the following is proved in ~\cite{NT}:
\begin{prop}\label{sl}
 The map $\Psi$ is nondegenerate. In particular one gets a rank-level duality isomorphism:
$$ \mathcal{V}_{\vect{\lambda}}(\mathfrak{X}, \mathfrak{sl}(r),s)\rightarrow \mathcal{V}^{\dagger}_{\vect{\gamma}}(\mathfrak{X}, \mathfrak{sl}(s),r).$$
\end{prop}

 We use Corollary ~\ref{ext} and the rank-level duality in ~\cite{R} to give an alternate proof of Proposition ~\ref{sl}. 
 \begin{Rem}
 A natural question is whether the rank-level dualities of ~\cite{NT} that arise from representation theory are same as the parabolic strange duality of ~\cite{R} obtained using geometry. The alternate proof of Proposition ~\ref{sl} in this paper gives an affirmative answer to the above question.
 \end{Rem}
\begin{Rem}
An alternative proof of rank-level dualities of ~\cite{NT} without using the result of ~\cite{R} can be obtained using the same strategy as the proof of rank-level duality result in ~\cite{M}. 
\end{Rem}
Next we consider the embedding $\phi :\mathfrak{sp}(2r)\oplus \mathfrak{sp}(2s) \rightarrow \mathfrak{so}(4rs)$. Let $\vect{\lambda} =(\lambda_1, \cdots, \lambda_n)$ be an $n$-tuple of dominant weights of $\mathfrak{sp}(2r)$ of level $s$. Then $\vect{\lambda}^T=(\lambda_1^T, \cdots, \lambda_n^T)$ is an $n$-tuple of dominant weights of $\mathfrak{sp}(2s)$ of level $r$. Assume that both $n$ and $\sum_{i=1}^n|\lambda_i|$ is even. We use the symplectic rank-level duality in ~\cite{A} and Corollary ~\ref{ext} to prove new symplectic rank-level dualities. 


\begin{cor}\label{ext5}
There is a linear isomorphism of the following spaces:
$$\mathcal{V}_{\vect{\lambda}}(\mathfrak{X}, \mathfrak{sp}(2r),s)\rightarrow  \mathcal{V}^{\dagger}_{\vect{{\lambda}}^T}(\mathfrak{X}, \mathfrak{sp}(2s),r),$$
\end{cor}
\subsection{Acknowledgments}
I am grateful to P. Belkale for his suggestions, comments, clarifications and constant encouragement. I thank S. Kumar for useful discussions. This paper is a part of my dissertation at UNC-Chapel Hill.

\section{Affine Lie algebras and conformal blocks }\label{conformalblock}

We recall some basic definitions from ~\cite{TUY} in the theory of conformal blocks. Let $\frg$ be a simple Lie algebra over $\mathbb{C}$ and $\mathfrak{h}$ a Cartan subalgebra of $\frg$. We fix the decomposition of $\frg$ into root spaces 
$$\frg=\mathfrak{h} \oplus \sum_{\alpha \in \Delta}\frg_{\alpha},$$ where $\Delta$ is the set of roots decomposed into a union of $\Delta_{+}\sqcup\Delta_{-}$ of positive and negative roots. Let $(,)$ denote the Cartan Killing form on $\frg$ normalized such that $(\theta, \theta)=2$, where $\theta $ is the longest root and we identify $\mathfrak{h}$ with $\mathfrak{h}^*$ using the form $(,)$.

\subsection{Affine Lie algebras} We define the affine Lie algebra $\widehat{\frg}$ to be 
$$\widehat{\frg}:= \frg\otimes \mathbb{C}((\xi)) \oplus \mathbb{C}c,$$ where $c$ belongs to the center of $\widehat{\frg}$ and the Lie bracket is given as follows:
$$[X\otimes f(\xi), Y\otimes g(\xi)]=[X,Y]\otimes f(\xi)g(\xi) + (X,Y)\Res_{\xi=0}(gdf).c,$$ where $X,Y \in \frg$ and $f(\xi),g(\xi) \in \mathbb{C}((\xi))$. 

Let $X(n)=X\otimes \xi^n$ and $X=X(0)$ for any $X \in \frg $ and $n \in \mathbb{Z}$. The finite dimensional Lie algebra $\frg$ can be realized as a subalgebra of $\widehat{\frg}$ under the identification of $X$ with $X(0)$. 

\subsection{Representation theory of affine Lie algebras} The finite dimensional irreducible modules of $\frg$ are parametrized by the set of dominant integral weights $P_{+} \subset \mathfrak{h}^*$. Let $V_{\lambda}$ denote the irreducible module of highest weight $\lambda \in P_{+}$ and $v_{\lambda}$ denote the highest weight vector.

We fix a positive integer $\ell$ which we call the level. The set of dominant integral weights of level $\ell$ is defined as follows 
$$P_{\ell}(\frg):=\{ \lambda \in P_{+} | (\lambda, \theta) \leq \ell\}$$
For each $\lambda \in P_{\ell}(\frg)$ there is a unique irreducible integrable highest weight $\widehat{\frg}$-module $\mathcal{H}_{\lambda}(\frg)$.
\subsection{Conformal blocks}We fix a $n$ pointed curve $C$ with formal neighborhood $\eta_1,\dots, \eta_n$ around the $n$ points $\vec{p}=(P_1,\dots, P_n)$, which satisfies the following properties :
\begin{enumerate}
\item The curve $C$ has at most nodal singularities,
\item The curve $C$ is smooth at the points $P_1,\dots, P_n$,
\item $C-\{P_1, \dots, P_n\}$ is an affine curve,
\item A stability condition (equivalent to the finiteness of the automorphisms of the pointed curve),
\item Isomorphisms $\eta_i:\widehat{\mathcal{O}}_{C,P_i}\simeq \mathbb{C}[[\xi_i]]$ for $i=1,\dots, n$. 
\end{enumerate}
We denote by $\mathfrak{X}=(C;\vec{p}; \eta_1,\dots,\eta_n)$ the above data associated to the curve $C$. We define another Lie algebra
$$\widehat{\frg}_n:=\bigoplus_{i=1}^n\frg\otimes_{\mathbb{C}}\mathbb{C}((\xi_i)) \oplus \mathbb{C}c,$$ where $c$ belongs to the center of $\widehat{\frg}_n$ and the Lie bracket is given as follows: 
$$[\sum_{i=1}^nX_i \otimes f_i, \sum_{i=1}^nY_i\otimes g_i]:=\sum_{i=1}^n[X_i,Y_i]\otimes f_ig_i + \sum_{i=1}^n(X_i, Y_i)\Res_{\xi_i=0}(g_idf_i)c.$$ 
We define the current algebra to be $$\frg(\mathfrak{X}):=\frg\otimes \Gamma(C-\{P_1,\dots, P_n\}, \mathcal{O}_{C}).$$ Consider an $n$-tuple of weights $\vec{\lambda}=(\lambda_1,\dots, \lambda_n) \in P_{\ell}^n(\frg)$. We set $\mathcal{H}_{\vec{\lambda}}=\mathcal{H}_{\lambda_1}(\frg)\otimes \dots \otimes \mathcal{H}_{\lambda_n}(\frg)$. The algebra $\widehat{\frg}_n$ acts on $\mathcal{H}_{\vec{\lambda}}$. For any $X\in \frg$ and $f\in \mathbb{C}((\xi_i))$, the action of $X\otimes f(\xi_i)$ on the $i$-th component is given by the following: 
$$\rho_i(X\otimes f(\xi_i))|v_1\otimes \dots \otimes v_n\rangle =|v_1\otimes \dots \otimes (X\otimes f(\xi_i)v_i) \otimes \dots \otimes v_n\rangle,$$ where $|v_i\rangle \in \mathcal{H}_{\lambda_i}(\frg)$ for each $i$. 
\begin{definition}
We define the space of conformal blocks 
$$\mathcal{V}^{\dagger}_{\vec{\lambda}}(\mathfrak{X}, \frg):=\operatorname{Hom}_{\mathbb{C}}(\mathcal{H}_{\vec{\lambda}}/\frg(\mathfrak{X})\mathcal{H}_{\vec{\lambda}}, \mathbb{C}).$$ 

We define the space of dual conformal blocks, $\mathcal{V}_{\vec{\lambda}}(\mathfrak{X}, \frg)=\mathcal{H}_{\vec{\lambda}}/\frg(\mathfrak{X})\mathcal{H}_{\vec{\lambda}}$. These are both finite dimensional $\mathbb{C}$-vector spaces which can be defined in families. The dimensions of these vector spaces are given by the {Verlinde formula}. 
\end{definition}

\section{Diagram automorphisms of symmetrizable Kac-Moody algebras.}

Consider a symmetrizable generalized Cartan Matrix A of size $n$ and let $\geg(A)$ denote the associated Kac-Moody algebra. To a symmetrizable generalized Cartan Matrix, one can associate a Dynkin diagram which is a graph on $n$ vertices, see [K] for details.
\begin{definition}
 A diagram automorphism of a Dynkin diagram is a graph automorphism i.e. it is a bijection $\omega$ from the set of vertices of the graph to itself such that for $1 \leq i,j \leq n$ the following holds:
$$a_{\omega i, \omega j} = a_{i,j}.$$

\end{definition}
\subsection{Outer Automorphisms}
We construct an automorphism of a symmetrizable Kac-Moody algebra $\geg(A)$ from a diagram automorphism $\omega$ of the Dynkin diagram of $\geg(A)$. The Kac-Moody Lie algebra $\geg(A)$ is the Lie algebra over $\CC$ generated by a Cartan subalgebra and the symbols $e_i$, $f_i$, $1\leq i \leq n$ with some relations. We refer the reader to ~\cite{K} for a detailed definition.

We start by defining the action of $\omega$ on the generators $e_i$ and $f_j$ in the following way: 
$$ \omega(e_{i}):=e_{\omega{i}} \ \ \mbox{and} \ \  \omega(f_{i}):=f_{\omega{i}}$$. For a simple coroot $\alpha_i^{\vee}$, we know that $\alpha_i^{\vee}=[e_i,f_i]$. Since $\omega$ is a Lie algebra automorphism, it implies the following:
$$\omega(\alpha_{i}^{\vee} )= \omega \left[ e_{i}, f_i \right] = \left[ e_{\omega i }, f_{\omega_{i}} \right]= \alpha_{\omega i }^{\vee},$$ where $\alpha_i^{\vee}$ are the simple coroots. 

In this way we have constructed an automorphism of the derived algebra $\geg(A)'=[\geg(A),\geg(A)]$. The extension of the action of $\omega$ to $\geg(A)$, follows from Lemma 1.3.1 in ~\cite{H}. The order of the automorphism $\omega$ of the Lie algebra $\mathfrak{g}(A)$ is same as the corresponding diagram automorphism. We will refer to these automorphisms as outer automorphisms. 

\subsection{Action $\omega^*$ on the affine fundamental weights} The map $\omega$ restricted to the Cartan subalgebra $\heh(A)$ defines an automorphism of $\omega : \heh(A) \rightarrow \heh(A)$. The adjoint action of $\omega^*$ on $\heh(A)^*$ is given by $\omega^*(\lambda)(x)=\lambda(\omega^{-1}x)$ for $\lambda \in \heh(A)^*$ and $x \in \heh(A)$. For $0\leq i \leq n$, let $\Lambda_i$ be the affine fundamental weight corresponding to the $i$-th coroot. Then $\omega^*(\Lambda_i)=\Lambda_{\omega i}$.

\section{Single-shift automorphisms of $\wgeg$}\label{conjugation}

Consider a finite dimensional complex simple Lie algebra $\geg$. For each $\alpha \in \Delta$, choose a non-zero element $X_{\alpha} \in \geg_{\alpha}$. Then, we have the following:
$$ 
[X_{\alpha}, X_{\beta}] = \begin{cases} 
N_{\alpha, \beta} X_{\alpha + \beta}  & \text{ if $\alpha + \beta \in \Delta$,} \\
0 & \text{if $\alpha + \beta \notin \Delta $},
\end{cases}
$$
where $N_{\alpha, \beta}$ is a non-zero scalar. The coefficients $N_{\alpha, \beta}$ completely determine the multiplication table of $\geg$. However, they depend on the choice of the elements $X_{\alpha}$. We refer the reader to ~\cite{Ser} for a proof of the following proposition:
\begin{prop}\label{chevbas}
One can choose the elements $X_{\alpha}$ is such a way so that 
$$ \left[ X_{\alpha}, X_{-\alpha} \right]=H_{\alpha} \ \mbox{ for all $\alpha \in \Delta$},$$
$$N_{\alpha, \beta}=-N_{-\alpha, -\beta} \ \mbox{ for all $\alpha$, $\beta$, $\alpha + \beta \in \Delta$}, $$

\end{prop}
where $H_{\alpha}$ is the coroot corresponding to $\alpha$. The basis $\{ X_{\alpha}, X_{-\alpha}, H_{\alpha} : \alpha \in \Delta_{+} \}$ is known as a Chevalley basis. 

To every $\alpha_i\in \Delta_{+}$, the simple coroot $X_{i} \in \heh$ is defined by the property $X_{i}(\alpha_j)=\delta_{ij}$ for $\alpha_j \in \Delta_{+}$. The lattice generated by $\{X_{i} : 1\leq i\leq \operatorname{rank}(\frg)\}$ is called the coweight lattice and is denoted by $P^{\vee}$. We identify $\heh$ with $\heh^*$ using the normalized Cartan killing form and let $h_{\alpha}$ denote the image of $\alpha$ under the identification.

For every $\mu \in P^{\vee}$, we define an map $\sigma_{\mu}$ of the Lie algebra $\wgeg$ 
\begin{eqnarray*} 
\sigma_{\mu} (c)&:=& c, \\
\sigma_{\mu}(h(n)) &:=&  h(n) + \delta_{n,0}\la \mu, h \ra .c \ \mbox{ where $h \in \heh$ and $n \in \mathbb{Z}$},\\ 
\sigma_{\mu}(X_{\alpha}(n)) &:=& X_{\alpha}(n + \la \mu, \alpha \ra).
\end{eqnarray*} 

\begin{prop} The map
$$ \sigma_{\mu} :\wgeg \rightarrow \wgeg,$$ is a Lie algebra automorphism. 
\end{prop} 
\begin{proof} 
We only need to verify that $\sigma_{\mu}$ is a Lie algebra homomorphism. It is enough to check that $\sigma_{\mu}$ respects the following relations:
\begin{eqnarray*}
   \left[H_{\alpha_i}(m), H_{\alpha_j}(n)\right] &=& \la H_{\alpha_i}, H_{\alpha_j} \ra. m.\delta_{m+n,0}c,\\
  \left[ H_{\alpha_i}(m), X_{\alpha}(n)\right] &=& \alpha(H_{\alpha_i})X_{\alpha}(m+n),\\
 \left[ X_{\alpha}(m), X_{\beta}(n)\right]&=& N_{\alpha, \beta}X_{\alpha + \beta}(m+n) \ \ \mbox{if $\alpha + \beta \in \Delta$},\\
 \left[X_{\alpha}(m), X_{-\alpha}(n)\right]  &=& H_{\alpha}(m+n) + \la X_{\alpha}, X_{-\alpha}\ra .m\delta_{m+n,0}c.
\end{eqnarray*}
It is trivial to see that $\sigma_{\mu}$ respects the first three relations. We only need to verify that $\sigma_{\mu}$ respects the last relation. Let us calculate the following:
\begin{eqnarray*} 
\left[ \sigma_{\mu}(X_{\alpha}(m)), \sigma_{\mu}(X_{-\alpha}(n)) \right] &=& \left[ X_{\alpha}(m + \la \mu, \alpha \ra), X_{-\alpha}(n-\la \mu, \alpha \ra)\right],\\
&=&  H_{\alpha}(m+n) + \la X_{\alpha}, X_{-\alpha} \ra . (m+ \la \mu, \alpha \ra ) \delta_{m+n,0}c.
\end{eqnarray*}  
If we apply $\sigma_{\mu}$ to the right hand side of the last relation, we get the following:
\begin{eqnarray*}
 \sigma_{\mu}(H_{\alpha}(m+n) &+& \la X_{\alpha}, X_{-\alpha}\ra .m\delta_{m+n,0}c)\\
 &=& H_{\alpha}(m+n) + (\la \mu, H_\alpha \ra  + \la X_{\alpha}, X_{-\alpha}\ra .m.)\delta_{m+n,0}c, \\
 &=& H_{\alpha}(m+n) + (\la \mu, \la X_{\alpha}, X_{-\alpha}\ra h_{\alpha}\ra  + \la X_{\alpha}, X_{-\alpha}\ra .m.)\delta_{m+n,0}c,\\
 &=& H_{\alpha}(m+n) + (\la X_{\alpha}, X_{-\alpha} \ra \la \mu, h_{\alpha} \ra + \la X_{\alpha}, X_{-\alpha}\ra .m.)\delta_{m+n,0}c,\\
 &=&  H_{\alpha}(m+n) + (\la X_{\alpha}, X_{-\alpha} \ra \la \mu, \alpha \ra + \la X_{\alpha}, X_{-\alpha}\ra .m.)\delta_{m+n,0}c,\\
 &=& \left[ \sigma_{\mu}(X_{\alpha}(m)), \sigma_{\mu}(X_{-\alpha}(n)) \right].
\end{eqnarray*}
This completes the proof. 
\end{proof}
The automorphism $\sigma_{\mu}$ was studied in~\cite{GGO}, ~\cite{VLW} and ~\cite{HF} and is called a single-shift automorphism. It is easy to observe that single-shift automorphisms are additive, i.e. for $\mu_1, \mu_2 \in P^{\vee}$, we have $\sigma_{\mu_1+\mu_2}=\sigma_{\mu_1}\circ\sigma_{\mu_2}$.
\subsection{Proof of Theorem \ref{conjugation}}
Let $G$ be a simply connected simple Lie group with Lie algebra $\geg$. We assume that $G$ is classical. Let $\mu \in P^{\vee}$, consider $\tau_{\mu}= \exp(\ln{\xi}.\mu)$. It is well defined upon a choice of a branch of the complex logarithm but conjugation by $\tau_{\mu}$ is independent of the branch of the chosen logarithm. We prove the following:
\begin{prop}\label{loopaut}
For $\mu \in P^{\vee}$, the map $x\rightarrow \tau_{\mu}x\tau_{\mu}^{-1}$ defines an automorphism $\varphi_{\mu}$ of $\tilde{\geg},$ where $\geg$ is a Lie algebra of type $A_n, B_n, C_n$ or $D_n$.
\end{prop}
 \begin{proof}
 
We only need to show that is that $\tau_{\mu}x \tau_{\mu}^{-1}$ has no fractional powers. Let $\mu_1$ and $\mu_2 \in P^{\vee}$, we first observe that $\tau_{\mu_1+\mu_2}= \tau_{\mu_1}\circ\tau_{\mu_2}$. Thus  it is enough to verify the claim for the fundamental coweights. For any classical Lie algebra, the roots spaces of $\geg$  are integral linear combinations of the matrix $E_{i,j},$ where $E_{i,j}$ is a matrix with $1$ at the $(i,j)$ entry and zero every where else. Therefore it is enough to show that $\tau_{\mu}E_{ij}(a)\tau_{\mu}^{-1}$ is of the form $E_{ij}(b)$, where $a$ and $b$ are integers. The rest of the proof is by direct computation.

Consider the case when $\geg$ is of type $A_n$. A basis for the coweight lattice is given by $X_i=\sum_{k=1}^i E_{k,k}- \frac{i}{n}\sum_{k=1}^n E_{k,k}$. 
In matrix notation, we have following:
$$\tau_{X_i}=\begin{bmatrix}
\xi^{(1-\frac{i}{n})} &0&  0   &\cdots &0 &0  & 0 \\
0 & \xi^{(1-\frac{i}{n})}&0  & \cdots&0 &0 & 0\\
0& 0&  \xi^{(1-\frac{i}{n})}& \cdots &0 &0 &0 \\
\vdots & \vdots & \vdots & \ddots &\vdots  &\vdots & \vdots \\
0      & 0      & 0      & \xi^{-\frac{i}{n}} &0 &0 & 0 \\
0& 0      & 0      & 0      & \xi^{-\frac{i}{n}} & 0 &0 \\
0 &0 & 0      & 0      & 0      & \xi^{-\frac{i}{n}}&0 \\
0& 0 &0 & 0      & 0      & 0      & \xi^{-\frac{i}{n}}.
\end{bmatrix}
$$ Thus for $a<b$ we get the following:

\[ \tau_{X_i}E_{a,b}(n)\tau_{X_i}^{-1}= \left \{ \begin{array}{ll}
                                       E_{a,b}(n) & \mbox{ if both $a$ and $b$ are less or greater than $i$}.\\
                                       E_{a,b}(n+1)  & \mbox{ if $a \leq i < b$}.                                                                          
\end{array} \right. \] The proof for $B_n$, $C_n$ and $D_n$ follows from a similar calculation.
\end{proof}
\begin{Rem}
 For $\mu \notin Q^{\vee}$, then it is clear that $\varphi_{\mu}$ is not a conjugation by a element of the loop group $G(\CC((\xi)))$.
\end{Rem}
The following Lemma compares the single-shift automorphisms with the automorphism $\varphi_{\mu}$ defined above. The proof follows directly by a easy computation.
\begin{Lemma}\label{loopeq}
For $\mu \in P^{\vee}$, the single-shift automorphism $\sigma_{\mu}$ on the loop algebra $\tilde{\geg}$ is same as $\varphi_{\mu}$
 \end{Lemma}
 
 \begin{proof}
 Since both the automorphism $\sigma_{\mu}$ and $\varphi_{\mu}$ are additive, it is enough to prove the above lemma for the simple coweights. It is easy to see that $\sigma_{\mu}$ and $\varphi_{\mu}$ are coincide on $\heh\otimes(\CC((\xi)))$. We just need to check for the equality on the non zero root spaces. This is an easy type dependent argument and follows from direct calculation.
 \end{proof}
 
 
We use Theorem ~\ref{conjugation} to prove the following for classical Lie algebras:
\begin{cor}\label{inner} 
If $\mu \in Q^{\vee}$ then $\sigma_{\mu}$ is an inner automorphism of $\wgeg$.
\end{cor}

\begin{proof}
First we observe that when $\mu \in Q^{\vee}$, the element $\tau_{\mu} \in G(\mathbb{C}((\xi)))$. Since $\tau_{\mu} \in G(\CC((\xi)))$, the map $\varphi_{\mu}$ can be realized as the adjoint action $\text{Ad}(\tau_{\mu})$ on $\tilde{\geg}$. By lemma ~\ref{loopeq}, we conclude that $\sigma_{\mu}$ is an inner automorphism of $\tilde{\geg}$

Let us denote the above conjugation on $G(\CC((\xi)))$ by $c(\tau_{\mu})$. We know that there is a central extension of the form
$$0 \rightarrow \CC^* \rightarrow \mathcal{G} \rightarrow G(\CC((\xi))) \rightarrow 0,$$ where $\mathcal{G}$ is a Kac-Moody group. We can take a any lift $\tilde{\tau}_{\mu}$ of $\tau_{\mu}$ in $\mathcal{G}$ and consider conjugation by $(\tilde{\tau}_\mu)$ on $\mathcal{G}$. The conjugation is well defined and independent of the lift. Since the extension $0 \rightarrow \CC c\rightarrow \wgeg \rightarrow \tilde{\geg} \rightarrow 0 $ is universal and central, it follows $\sigma_{\mu}$ is of the form $\text{Ad}(\tilde{\tau}_{\mu})$. This completes the proof.
\end{proof}

\section{Extension of single-shift automorphisms} 
 Let $\geg_1$, $\geg_2$ and $\geg$ be simple Lie algebras and $\heh_1$, $\heh_2$ and $\heh$ be their Cartan subalgebras. For $i \in \{1,2\}$, let $\phi_i:\geg_i\rightarrow \geg$ be an embedding of Lie algebras such that $\phi_i(\heh_i)\subset \heh$. Further let us denote the normalized Cartan Killing form on $\geg_1$, $\geg_2$ and $\geg$ by $\la \ , \ \ra_{\geg_1}$, $\la \ , \ \ra_{\geg_2}$, $\la \ , \mbox{and}  \ \ra_{\geg}$ respectively, and let $(\ell_1, \ell_2)$ be the Dynkin multi-index of the embedding $\phi=(\phi_1, \phi_2)$. We can extend the map $\phi=(\phi_1,\phi_2)$ to a map $\widehat{\phi}$ of $\widehat{\geg}_1\oplus \widehat{\geg}_2 \rightarrow \widehat{\geg}$ as follows:
\begin{eqnarray*}
\widehat{\phi_1}: \wgeg_1 & \rightarrow &  \wgeg, \\
X\otimes f + a.c &\rightarrow & \phi_1(X)\otimes f + a.\ell_1.c,
\end{eqnarray*}where $X \in \geg_1$, $f \in \CC((\xi))$ and $a$ is a constant. We similarly map 
\begin{eqnarray*}
\widehat{\phi_2}: \wgeg_2 & \rightarrow & \wgeg, \\
Y\otimes g + b.c & \rightarrow & \phi_2(Y)\otimes g + b.\ell_2.c, 
\end{eqnarray*} where $Y \in \geg_2$, $g \in \CC((\xi))$ and $b$ is a constant. We define $\widehat{\phi}$ to be $\widehat{\phi}_1+ \widehat{\phi}_2$. Consider an element $\mu \in P_1^{\vee}$ such that $\tilde{\mu}=\phi_1(\mu) \in P^{\vee}$, where $P_1^{\vee}$ and $P$ denote the coweight lattices of $\geg_1$ and $\geg$ respectively.

Let $\alpha$ be a root in $\geg_1$ with respect to a Cartan subalgebra $\heh_1$ and $X_{\alpha}$ be a non-zero element of $\frg_1$ in the root space ${\alpha}$. If $\phi_1(X_{\alpha})= \sum_{i=1}^{\dim{\heh}}a_ih_i + \sum_{\gamma \in I_{\alpha}}a_{\gamma}X_{\gamma}$, where $h_i$'s be any basis of $\heh$. We have the following lemma:

\begin{Lemma}\label{case1} 
For all $i$, we claim that $a_i=0$.

\end{Lemma}
\begin{proof}
For any element $h \in \heh_1$, we consider the following Lie bracket. 
\begin{eqnarray*}
\left[\phi(h), \phi(X_{\alpha}\right)]&=& \phi(\left[ h , X_{\alpha}\right]),\\
&=& \phi(\alpha(h)X_{\alpha}),\\
&=& \sum_{i=1}^{\dim{\heh}}a_i\alpha(h)h_i+ \sum_{\gamma \in I_{\alpha}}a_{\gamma}\alpha(h)X_{\gamma}.
\end{eqnarray*}
On the other hand $\left[\phi(h), \phi(X_{\alpha})\right]= \sum_{\gamma \in I_{\alpha}}a_{\gamma}\gamma(\phi(h))X_{\gamma}$. Comparing, we see that $a_i=0$ for all $i$ and $\gamma(\phi(h))=\alpha(h)$ for all $h$ in $\heh_1$. 
\end{proof}

Next, we prove the following lemma: 
\begin{Lemma}\label{case2}
If $\gamma \in I_{\alpha}$, then for all $h_2 \in \heh_2$ $$\gamma(\phi(h_2))=0.$$ 
\end{Lemma}

\begin{proof}
 For $h_2 \in \heh_2$, we have the following:
\begin{eqnarray*}
\phi\left[ h_2, X_{\alpha}\right]&=& \left[\phi(h_2), \phi(X_{\alpha})\right],\\
&=& \sum_{\gamma \in I_{\alpha}} a_{\gamma}\gamma(\phi(h_2))X_{\gamma},\\
&=& 0.
\end{eqnarray*}
Thus comparing, we get $\gamma(\phi(h_2))=0$ for all $h_2 \in \heh_2$.
\end{proof}

The following proposition is about the extension of single-shift automorphisms:
\begin{prop}{\label{key}} The single-shift automorphism $\sigma_{\tilde{\mu}}$ restricted to $\widehat{\phi_1}(\wgeg_1)$ is the automorphism $\sigma_{\mu}$. Moreover $\sigma_{\tilde{\mu}}$ restricts to identity on $\widehat{\phi_2}(\wgeg_2)$.
\end{prop}

\begin{proof} Let $n$ be an integer and $h$, $h_1$ be elements of $\heh_1$ and $\heh_2$. We need to show the following identities:
$$\sigma_{\tilde{\mu}}(\widehat{\phi_1}(h_1(n)))=\widehat{\phi_1}(\sigma_{\mu}(h_1(n))), $$
$$\sigma_{\tilde{\mu}}(\widehat{\phi_1}(X_{\alpha}(n)))=\widehat{\phi_1}(\sigma_{\mu}(X_{\alpha}(n))),$$
$$\sigma_{\tilde{\mu}}(\widehat{\phi_2}(h_2(n)))=\widehat{\phi_2}((h_2(n))),$$
$$\sigma_{\tilde{\mu}}(\widehat{\phi_2}(X_{\beta}(n)))=\widehat{\phi_2}((X_{\beta}(n))),$$ 
where $\alpha$, $\beta$ are roots of $\geg_1$, $\geg_2$ respectively, and $X_{\alpha}$, $X_{\beta}$ are non-zero elements in the root space of $\alpha$, $\beta$ respectively. Let $h \in \heh_1$, we have the following:
\begin{eqnarray*}
\sigma_{\tilde{\mu}}(\widehat{\phi_1}(h(n)))&=& \widehat{\phi_1}(h(n))+ \delta_{n,0}.\la \tilde{\mu}, \phi_1(h)\ra_{\geg} .c,\\
&=& \widehat{\phi_1}(h(n)) + \delta_{n,0} .\la \phi_1(\mu), \phi_1(h)\ra_{\geg}.c,\\
&=& \widehat{\phi_1}(h(n)) + \delta_{n,0}.\ell_1.\la \mu, h \ra_{\geg_1}.c,\\
&=& \widehat{\phi_1}(\sigma_{\mu}(h(n))).
\end{eqnarray*}
This completes the proof of the first identity. For the second identity, we use Lemma ~\ref{case1}. For any non-zero element $X_{\alpha}$ in the root space of $\alpha$, consider the following:
\begin{eqnarray*}
\widehat{\phi_1}(\sigma_{\mu}(X_{\alpha}(n)))&=& \widehat{\phi_1}(X_{\alpha}(n + \alpha(\mu))),\\
&=& \sum_{\gamma \in I_{\alpha}}a_{\gamma}X_{\gamma}(n+ \alpha(\mu)),\\
&=& \sum_{\gamma \in I_{\alpha}}a_{\gamma}X_{\gamma}(n + \gamma(\phi(\mu))),\\
&=& \sigma_{\tilde{\mu}}(\widehat{\phi_1}(X_{\alpha}(n))).
\end{eqnarray*}

 To prove the fourth identity we use Lemma ~\ref{case2}. For a non-zero element $X_{\beta}$ in the root space ${\beta}$, consider the following:
 \begin{eqnarray*}
\sigma_{\tilde{\mu}}(\widehat{\phi}_2(X_{\beta}(n)))&=& \sigma_{\phi_1(\mu)}(\sum_{\gamma \in I_{\beta}}X_{\gamma}(n)),\\
&=& \sum_{\gamma \in I_{\beta}}\sigma_{\phi_1(\mu)}(X_{\gamma}(n)),\\
&=& \sum_{\gamma \in I_{\beta}} X_{\gamma}(n + \gamma(\phi(\mu))),\\
&=& \sum_{\gamma \in I_{\beta}} X_{\gamma}(n),\\
&=& \widehat{\phi_2}(X_{\beta}(n)).
\end{eqnarray*}
We are only left to show that $\sigma_{\tilde{\mu}}(\widehat{\phi_2}(h_2(n)))=\widehat{\phi_2}(h_2(n))$ for $h_2\in \heh_2$, which follows from the following lemma. 
\end{proof}

\begin{Lemma}
$\la \phi_1(h_1), \phi_2(h_2)\ra_{\geg} =0$ for any element $h_1$, $h_2$ of $\heh_1$ and $\heh_2$ respectively.
\end{Lemma}
\begin{proof} It is enough to proof the result for all $H_{\beta}$, where $\beta$ is a root of $\geg_2$. Let $\phi(X_{\beta})=\sum_{\lambda \in I_{\beta}}a_{\lambda}X_{\lambda}$ and $\phi(X_{-\beta})=\sum_{\gamma\in I_{-\beta}}a_{\gamma}X_{\gamma}$. Since $\phi_2(\heh_2) \subset \heh$, we get 
$\phi(H_{\beta})=\sum_{\gamma \in I_{\beta}\cap(-I_{-\beta})}a_{\gamma}b_{-\gamma}H_{\gamma}$. 

Now $\left[ \phi(h_1), \phi(X_{\beta})\right]= \sum_{\gamma \in I_{\beta}} a_{\gamma} \gamma(\phi(h_1))X_{\gamma}= 0$. Thus for $\gamma \in I_{\beta}$, we get $\gamma(\phi(h_1))=\la h_\gamma, \phi(h_1)\ra_{\geg}=0$ which implies $\la \phi_1(h_1), H_{\gamma}\ra_{\geg}=0$. Thus, we have the following:

\begin{eqnarray*}
\la \phi_1(h_1), \phi_2(H_{\beta})\ra_{\geg} &=& \sum_{\gamma \in I_{\beta}\cap(-I_{-\beta})}a_{\gamma}b_{-\gamma}\la \phi_1(h_1), H_{\gamma} \ra_{\geg},\\
&=& 0.
\end{eqnarray*}

\end{proof}

\section{Multi-shift automorphisms} 

We recall the definition of multi-shift automorphisms following ~\cite{FS}. Let us fix a sequence of pairwise distinct complex numbers $z_s$ for $s \in \{1, \cdots, n\}$, the coroot $H_{\alpha}$ corresponding to the roots $\alpha$. Let $P^{\vee}$ and $Q^{\vee}$ denote the coweight and the coroot lattice of $\geg$ respectively and $\Gamma_n^{\geg}=\{(\mu_1,\dots,\mu_n) | \mu_i \in P^{\vee} \text{and} \sum_{i=1}^n\mu_i=0\}$. Consider a Chevalley basis given by $\{ X_{-\alpha}$, $X_{\alpha}$, $H_{\alpha} : \alpha \in \Delta_{+} \}$. For $\vect{\mu} \in \Gamma_{n}^{\geg}$, we define a multi-shift automorphism $\sigma_{\vect{\mu},t}(\vec{z})$ of $\wgeg$ as follows:
\begin{eqnarray*}
\sigma_{\vect{\mu},t}(\vec{z})(c)&:=&(c),\\
\sigma_{\vect{\mu},t}(\vec{z})(h)\otimes f&:=& h\otimes f + \bigg(\sum_{s=1}^n\la h, \mu_s\ra \Res(\varphi_{t,s}.f)\bigg).c, \\
\sigma_{\vect{\mu},t}(\vec{z})(X_{\alpha}\otimes f)&:=& X_{\alpha}\otimes f. \prod_{s=1}^n{\varphi_{t,s}^{-\alpha(\mu_s) }},
\end{eqnarray*}
where $f \in \CC((\xi))$, $\varphi_{t,s}(\xi)=(\xi+ (z_t-z_s))^{-1}, h \in \heh.$

Let us now recall some important properties of the multi-shift automorphisms.
\begin{enumerate}
\item The multi-shift automorphism $\sigma_{\vec{\mu},t}(\vec{z})$ has the same outer automorphism class as the single shift automorphism $\sigma_{\mu_t}$.
\item It is shown in [FS] that $\sigma_{\vect{\mu},t}$ is a Lie algebra automorphism of $\wgeg$ and can be easily extended to an automorphism of $\wgeg_{n}$.

\item Multi-shift automorphisms of $\wgeg_n$ preserve the current algebra $\geg\otimes \Gamma(\mathbb{P}^1-\vec{p},\mathcal{O}_{\mathbb{P}^1}),$ where $\vec{p}=(P_1,\dots, P_n)$ are $n$ distinct points with coordinates $z_1,\dots, z_n$. 

\end{enumerate}
The following is one of the main results of ~\cite{FS}
\begin{prop}\label{fs}
Let $\mathfrak{X}$ be the data associated to $n$ distinct points $\vec{z}$ on $\mathbb{P}^1$ with chosen coordinates. Then there is an isomorphism 
$$\Theta_{\vec{\mu}}(\vec{z}) : \mathcal{V}_{\vec{\lambda}}(\mathfrak{X}, \frg, \ell)\rightarrow \mathcal{V}_{\vec{\sigma}\vec{\lambda}}(\mathfrak{X}, \frg, \ell).$$ More over the isomorphism is flat with respect to the KZ/Hitchin/WZW connection.
\end{prop}
\begin{Rem}It is easy to observe that to prove Theorem \ref{extension1} we need to show the same identities that we showed in the proof of Proposition \ref{key}. Hence the proof of Theorem \ref{extension1} follows from the proof of Proposition \ref{key}. 
\end{Rem}
\subsection{Multi-shift automorphisms as conjugations}
We restrict to the case $\geg$ is of type $A_n$, $B_n$, $C_n$ or $D_n$. It is easy to see that for $\vect{\mu} \in \Gamma_{N}^{\mathfrak{g}}$, the multi-shift automorphism $\sigma_{\vect{\mu}, s}$ of $\wgeg$ descends to an automorphism $\sigma_{\vect{\mu},t}$ of $\tilde{\geg}=\geg\otimes \CC((\xi))$. Let $\vect{\mu}\in \Gamma_n^{\geg}$ and consider $n$ distinct points $P_1, P_2, \cdots, P_n$ on $\mathbb{P}^1-\infty$ and denote their coordinates by $\xi_i(P_i)=z-z_i$, where $z$ is a global coordinate of $\CC$. We now consider $\tau_{\vect{\mu}}=\exp(\ln {(\xi_1)}\mu_1 + \ln {(\xi_2)}\mu_2 + \cdots + \ln {(\xi_n)}\mu_n)$ which is only well defined up to a choice of a branch of the logarithms. We can rewrite $\xi_i=z-z_i$ as $\xi_t+z_t-z_i$ and expand $\exp{(\ln(\xi_s)\mu_s)}$ in terms of $\xi_t$. We rewrite $\tau_{\vect{\mu}}$, in the coordinate $\xi_t$ and we rename it as $\tau_{\vect{\mu}}$ as $\tau_{\vect{\mu},t}$. We can conjugate by $\tau_{\vect{\mu},t}$ on $\tilde{\geg}=\geg\otimes \CC((\xi_t))$. Let us denote the conjugation by $c(\tau_{\vect{\mu},t})$. It is well defined and is independent of the branch of the logarithm chosen. We have the following proposition, the proof of which follows by a direct calculation:
\begin{prop}
 The automorphisms $c(\tau_{\vect{\mu},t})$ and $\sigma_{\vect{\mu},t}$ coincide on $\tilde{\geg}$.
\end{prop}

\begin{Rem}
If one of the chosen points $P_i$ is infinity, the formula of the multi-shift automorphism needs a minor modification to accommodate the new coordinate at infinity. This has been considered in ~\cite{FS}. 
\end{Rem}

\section{New rank-level dualities for $\mathfrak{sp}(2r)$}
In this section, we use Corollary 2 and symplectic strange duality of T. Abe ~\cite{A} to generate new symplectic rank-level dualities. 

\subsection{Young Diagrams}Let $\mathcal{Y}_{r,s}$ denote the set of Young diagrams with at most $r$ rows and $s$ columns. For a Young diagram $Y=(a_1,a_2,\cdots, a_r) \in \mathcal{Y}_{r,s}$, we denote by $Y^T$ the Young diagram obtained by exchanging the rows and columns.
 For a Young diagram $Y \in \mathcal{Y}_{r,s}$, we denote by $Y^c$ the Young diagram given by the conjugate $(s-a_r, s-a_{r-1}, \cdots, s-a_1)$. The Young diagram $Y^*$ is defined to be $({Y^{T}})^c$. It is easy to see $(Y^T)^c=(Y^c)^T$, where $Y \in \mathcal{Y}_{r,s}$.

There is a one to one correspondence between $P_s(\mathfrak{sp}(2r))$ and Young diagrams $\mathcal{Y}_{r,s}$. For $\lambda \in P_s(\mathfrak{sp}(2r))$, the corresponding Young diagram is denoted by $Y(\lambda)$. The dominant weight of $\mathfrak{sp}(2s)$ of level $r$, corresponding to the Young diagram $Y(\lambda)^*$ will be denoted by $\lambda^*$ and that of $Y(\lambda)^T$ by $\lambda^T$. It is easy to observe that $\lambda \rightarrow \lambda^*$ gives a bijection of $P_s(\mathfrak{sp}(2r))$ with $P_r(\mathfrak{sp}(2s))$. Also $\lambda \rightarrow \lambda^T$ gives a bijection of $P_s(\mathfrak{sp}(2r))$ with $P_r(\mathfrak{sp}(2s))$.

\subsection{Action of Center}We now describe the action of the center of $Sp(2r)$ as diagram automorphisms on $P_s(\mathfrak{sp}(2r))$. Let $\omega$ be the outer automorphism that corresponds to the diagram automorphism which sends the $i$-th vertex to $r-i$-th vertex of the Dynkin diagram of $\widehat{\mathfrak{sp}}(2r)$, where $0\leq i \leq r$. Then the Young diagram of $\omega^*\lambda$ is given by $Y(\lambda)^c$, where $Y(\lambda)$ is the Young diagram corresponding to $\lambda \in P_{s}(\mathfrak{sp}(2r))$. 

The action of the non-trivial element $\omega \in Z(\operatorname{Sp}(2s))$ on $P_{r}(\mathfrak{sp}(2s))$ gives us the following:
\begin{Lemma}
For $\lambda \in \mathcal{Y}_{r,s}$, we get 
$$\omega (\lambda^*)=\lambda^T.$$
\end{Lemma}
\subsection{New symplectic rank-level duality}

Let us fix $n$ distinct smooth points $\vec{p}=(P_1, \cdots, P_n)$ on the projective line $\mathbb{P}^1$. Let $\vec{z}=(z_1,\dots, z_n)$ be the local coordinates of $\vec{p}$. We denote the above data by $\mathfrak{X}$. Consider an $n$-tuple of level $s$ weights $\vect{\lambda}=(\lambda_1, \cdots, \lambda_n)$ and $\vect{{\lambda}}^*:=(\lambda_1^*, \cdots,\lambda_n^*)$. Assume that both $n$ and $\sum_{j=1}^n|Y(\lambda_i)|$ are even, the following is the main result in ~\cite{A}:
\begin{prop}\label{abeasd}
The rank-level duality map 
$$\mathcal{V}_{\vect{\lambda}}(\mathfrak{X}, \mathfrak{sp}(2r),s)\rightarrow\mathcal{V}^{\dagger}_{\vect{\lambda}^*}(\mathfrak{X}, \mathfrak{sp}(2s),r),$$ is an isomorphism. 
\end{prop}
From the branching rule described in ~\cite{H}, we also get a map 
$$\Psi :\mathcal{V}_{\vect{\lambda}}(\mathfrak{X}, \mathfrak{sp}(2r),s)\otimes \mathcal{V}_{\vect{{\lambda}}^T}(\mathfrak{X}, \mathfrak{sp}(2s),r) \rightarrow \mathcal{V}_{\vect{\Lambda}}(\mathfrak{X}, \mathfrak{so}(4rs),1),$$ 
where $\vec{\Lambda}=(\Lambda_1, \dots, \Lambda_n)$ and $\Lambda_i$ is the unique level one dominant weight of $\mathfrak{so}(4rs)$ such that $\lambda_i$ and $\lambda_i^T$ appear in the branching of $\Lambda_i$. Assume that both $n$ and $\sum_{i=1}^n|\lambda_i|$ are even. We choose $\vec{\mu} \in \Gamma_n^{\mathfrak{sp}(2s)}$ such that for $1\leq i \leq n$, $\mu_i \in P(\mathfrak{sp}(2s))^{\vee}\backslash Q(\mathfrak{sp}(2s))^{\vee}$, and use Proposition \ref{abeasd} and Corollary 2 to get new symplectic rank-level dualities. 
\begin{prop}\label{ext5}
There is a linear isomorphism of the following spaces:
$$\mathcal{V}_{\vect{\lambda}}(\mathfrak{X}, \mathfrak{sp}(2r),s)\rightarrow  \mathcal{V}^{\dagger}_{\vect{{\lambda}}^T}(\mathfrak{X}, \mathfrak{sp}(2s),r).$$
\end{prop}

\section{Branching rules and rank-level dualities for $\mathfrak{sl}(r)$}

We describe the branching rules of the embedding $\mathfrak{sl}(r)\oplus \mathfrak{sl}(s) \subset \mathfrak{sl}(rs)$ following ~\cite{ABI}.
Let $P_{+}(\mathfrak{sl}(r))$ denote the set of dominant integral weights of $\mathfrak{sl}(r)$ and $\Lambda_1, \cdots, \Lambda_{r-1}$ denote the fundamental weights of $\mathfrak{sl}(r)$. If $\lambda=\sum_{i=1}^{r-1}\til{k_i}\Lambda_i$ for non-negative integers $\til{k}_i$, we rewrite $\lambda$ as
${\lambda}= \sum_{i=0}^{r-1}\til{k_i}\Lambda_{i},$ where $ \sum_{i=0}^{r-1}\til{k_i}=s,$ and $\Lambda_0$ is the affine $0$-th fundamental weight. The level one weights of $\mathfrak{sl}(rs)$  are given by ${P}_{1}(\mathfrak{sl}(rs))= \{ \Lambda_{0}, \cdots, \Lambda_{rs-1} \}.$ We identify it with the set $\{1,2, \cdots, rs-1\}$.  
\subsection{Action of center and branching rule}
Let $\widehat{\rho} = g^*(\mathfrak{sl}(r))\Lambda_{0}+ \frac{1}{2} \sum_{\alpha \in \Delta_{+}} \alpha, $ where $g^*(\mathfrak{sl}(r))$ is the dual Coxeter number of $\mathfrak{sl}(r)$ and $\Delta_{+}$ is the set of positive roots respect to a chosen Cartan subalgebra of $\mathfrak{sl}(r)$. Consider ${\lambda} + \widehat{\rho} = \sum_{i=0}^{r-1} k_i \Lambda_i,$ where $k_i = \til{k_i}+1 $ and $\sum_{i=0}^{r-1}k_i={r+s}$. The center of $\operatorname{SL}(r)$ is $\ZZ/r\ZZ$. The action of $\ZZ/ r\ZZ$ induced from outer automorphisms on $P_{s}(\mathfrak{sl}(r))$ is described as follows:
\begin{eqnarray*} 
\mathbb{Z}/r\mathbb{Z} \times {P}_{s}(\mathfrak{sl}(r)) & \longrightarrow & {P}_{s}(\mathfrak{sl}(r)), \\
(\sigma, \Lambda_{i}) & \longrightarrow & \Lambda_{((i + \sigma)mod(r))}. 
\end{eqnarray*}
Let $\Omega_{r,s}= {P}_{s}(\mathfrak{sl}(r))/ ({\mathbb{Z}/r\mathbb{Z}})$ be the set of orbits under this action and similarly let $\Omega_{s,r}$ be the orbits of ${P}_{r}(\mathfrak{sl}(s))$ under the action of $\mathbb{Z}/s\mathbb{Z}$. The following map $\beta$ parametrizes the bijection.
$$ \beta: {P}_{s}(\mathfrak{sl}(r))\rightarrow {P}_{r}(\mathfrak{sl}(s))$$ 
Set $a_j = \sum_{i=j}^r k_i, \ \mbox{for}\   1 \leq j \leq r \ \mbox{and}  \ k_r=k_0.$ The sequence $\vect{a}=(a_1, a_2, \cdots, a_r)$ is decreasing. Let $(q_1, q_2, q_3, \cdots, q_s)$ be the complement of $\vect{a}$ in the set $\{ 1, 2, \cdots, (r+s)\}$ in decreasing order. We define the following sequence:
$$b_j= r+s + q_s -q_{s-j+1} \ \mbox{for} \ 1 \leq j \leq s. $$ The sequence $b_j$ defined above is also decreasing. The map $\beta$ is given by the following formula: 
$$\beta(a_1, \cdots, a_r) = (b_1, b_2, \cdots, b_s).$$

Thus when ${\lambda}$ runs over an orbit of $\Omega_{r,s}$, ${\gamma}=\sigma.\beta({\lambda})$ runs over an orbit of $\Omega_{s,r}$ if $\sigma$ runs over $\mathbb{Z}/s\mathbb{Z}$. 

The elements $\lambda$ of ${P}_{s}(\mathfrak{sl}(r))$ can be parametrized by Young diagrams $Y({\lambda})$ with at most $r-1$ rows and at most $s$ columns. Let $Y({\lambda})^{T}$ be the modified transpose of $Y({\lambda})$. If $Y({\lambda})$ has rows of length $s$, then $Y({\lambda})^T$ is obtained by taking the usual transpose of $Y({\lambda})$ and deleting the columns of length $s$. We denote by $\lambda^T$ the dominant integral weight of $\mathfrak{sl}(s)$ of level $r$ that corresponds to $Y({\lambda})^{T}$. With this notation we recall the following proposition from ~\cite{ABI}:

\begin{prop}
Let ${\lambda} \in {P}_{s}(\mathfrak{sl}(r))$ and $c({\lambda})$ be the number of columns of $Y({\lambda})$. Suppose $\sigma= c({\lambda})\operatorname{mod}{s}$. Then $ \sigma.\beta({\lambda})= {\lambda}^{T}.$
\end{prop}

 The following Proposition from ~\cite{ABI} describes the multiplicity $m^{\Lambda}_{\lambda, \gamma}$ of the component $\mathcal{H}_{\lambda}\otimes \mathcal{H}_{\gamma}$ is $\mathcal{H}_{\Lambda}$. 
 \begin{prop} \label{branching1}
 For $\lambda \in P_{s}(\mathfrak{sl}(r))$ and $\sigma \in \ZZ/s\ZZ$, let $\delta({\lambda}, \sigma)=|Y({\lambda})| + rs + r(\sigma - c(Y({\lambda})))$. The multiplicity $m^{\Lambda}_{\lambda,\gamma}$ is given by the following formula:
 \begin{eqnarray*}
 m^{\Lambda}_{{\lambda}, {\gamma}} &=&1 \ \mbox{if} \ {\gamma}= \sigma\beta({\lambda}), \ \ \sigma \in \mathbb{Z}/s\mathbb{Z} \ \mbox{and} \ \Lambda= \delta({\lambda}, \sigma)\operatorname{mod}(rs),\\
 m^{\Lambda}_{{\lambda},{\gamma}}&=&0\ \mbox{otherwise}.
 \end{eqnarray*}
 \end{prop}

\subsection{Rank-level duality of $\mathfrak{sl}(r)$}

Let us fix $n$ distinct points $\vec{p}=(P_1, \cdots, P_n)$ on the projective line $\mathbb{P}^1$. Let $\vec{z}=(z_1,\dots, z_n)$ be the local coordinates of $\vec{p}$. We denote the above data by $\mathfrak{X}$. Consider an $n$-tuple $\vec{\Lambda}=(\Lambda_{i_1}, \cdots, \Lambda_{i_n})$ of level one dominant integral weights of $\mathfrak{sl}(rs)$. Let $\vect{\lambda}=(\lambda_1, \cdots, \lambda_n)$ and $\vec{\gamma}=(\gamma_1,\dots,\gamma_n)$ be such that $\lambda_k, \gamma_k$ appear in the branching of $\Lambda_{i_k}$ for $1\leq k \leq n$. We get a map
\begin{equation*}\label{sd1}
\mathcal{V}_{\vect{\lambda}}(\mathfrak{X}, \mathfrak{sl}(r),s)\otimes \mathcal{V}_{\vect{\gamma}}(\mathfrak{X}, \mathfrak{sl}(s),r)\rightarrow \mathcal{V}_{\vect{\Lambda}}(\mathfrak{X}, \mathfrak{sl}(rs),1),
\end{equation*}
If $rs$ divides $\sum_{k=1}^n{i_k}$, it is well known that $\dim_{\CC}\mathcal{V}_{\vect{\Lambda}}(\mathfrak{X}, \mathfrak{sl}(rs),1)=1$. We get the following morphism well defined up to scalars:
$$\Psi:\mathcal{V}_{\vect{\lambda}}(\mathfrak{X}, \mathfrak{sl}(r),s) \rightarrow \mathcal{V}^{\dagger}_{\vect{\gamma}}(\mathfrak{X}, \mathfrak{sl}(s),r).$$ 

The rest of the section is devoted to the proof that $\Psi$ is an isomorphism. Without loss of generality we can assume that $\sigma_i-c(Y(\lambda_i))$ is non-negative for all $i$. Let $Q_1$ be a new point distinct from $P_1, \dots, P_n$ on $\mathbb{P}^1$ and $\eta_1$ be the new coordinate. Let $\widetilde{\mathfrak{X}}$ be the data associated to the points $P_1,\dots, P_n, Q_1$ on $\mathbb{P}^1$. We have the following proposition:
\begin{prop}
The following are equivalent:
\begin{enumerate}
\item The rank-level duality map 
$$\mathcal{V}_{\vect{\lambda}}({\mathfrak{X}}, \mathfrak{sl}(r),s) \rightarrow  \mathcal{V}^{\dagger}_{\vect{\gamma}}({\mathfrak{X}}, \mathfrak{sl}(s),r),$$ is an isomorphism.
\item The rank-level duality map $$\mathcal{V}_{\vect{\lambda}\cup{0}}(\widetilde{\mathfrak{X}}, \mathfrak{sl}(r),s) \rightarrow \mathcal{V}^{\dagger}_{\vect{\gamma}\cup{0}}(\widetilde{\mathfrak{X}}, \mathfrak{sl}(s),r),$$ is an isomorphism.
\item The rank-level duality map 
$$\mathcal{V}_{\vect{\lambda}\cup{0}}(\widetilde{\mathfrak{X}}, \mathfrak{sl}(r),s) \rightarrow  \mathcal{V}^{\dagger}_{\vect{\lambda}^T\cup\beta}(\widetilde{\mathfrak{X}}, \mathfrak{sl}(s),r),$$ is an isomorphism, 
where $\beta=r\omega_{\sigma}$, $\sigma=\sum_{i=1}^n(\sigma_i-c(Y(\lambda_i)))\operatorname{mod}(s)$ and $\omega_{\sigma}$ is the $\sigma$-th fundamental weight.
\end{enumerate}
\end{prop}\label{right}
\begin{proof}
The equivalence of $(1)$ and $(2)$ follows from the compatibility of propagation of vacua (see ~\cite{TUY}) with rank-level duality. The equivalence of $(2)$ and $(3)$ follows directly from Corollary 2.
 
\end{proof}

The proof that $\Psi$ is non-degenerate follows from the following (see Theorem 4.10 ~\cite{R}):
\begin{prop}The following rank-level duality map is an isomorphism
$$\mathcal{V}_{\vect{\lambda}\cup{0}}(\widetilde{\mathfrak{X}}, \mathfrak{sl}(r),s) \rightarrow \mathcal{V}^{\dagger}_{\vect{\lambda}^T\cup\beta}(\widetilde{\mathfrak{X}}, \mathfrak{sl}(s),r),$$ 
where $\beta=r\omega_{\sigma}$, $\sigma=\sum_{i=1}^n(\sigma_i-c(Y(\lambda_i)))\operatorname{mod}(s)$ and $\omega_{\sigma}$ is the $\sigma$-th fundamental weight. 
\end{prop}

\end{document}